\documentclass[12pt]{amsart}
\usepackage{amsmath,amssymb,amsbsy,amsfonts,amsthm,latexsym,
                        amsopn,amstext,amsxtra,euscript,amscd,mathrsfs,color,bm}
                        
\usepackage{float} 
\usepackage[english]{babel}
\usepackage{mathtools}
\usepackage{todonotes}
\usepackage{url}
\usepackage[colorlinks,linkcolor=blue,anchorcolor=blue,citecolor=blue,backref=page]{hyperref}

\usepackage[norefs,nocites]{refcheck}

\newtheorem{theorem}{Theorem}
\newtheorem{lemma}[theorem]{Lemma}

\numberwithin{equation}{section}
\numberwithin{theorem}{section}

\newcommand{\cC}{\mathcal{C}}
\newcommand{\cK}{\mathcal{K}}
\newcommand{\e}{\mathbf{e}}
\newcommand{\C}{\mathbb{C}}
\newcommand{\F}{\mathbb{F}}
\newcommand{\Z}{\mathbb{Z}}
\newcommand{\Q}{\mathbb{Q}}
\newcommand{\R}{\mathbb{R}}
\newcommand{\Arg}{\mathrm{Arg}}
\newcommand{\Tr}{\mathrm{Tr}}
\newcommand{\h}{\mathrm{h}}

\def\le{\leqslant}
\def\ge{\geqslant}

\def\({\left(}
\def\){\right)}

\def\rf#1{\left\lceil#1\right\rceil}

\begin{document}

\title[M{\"o}bius randomness law for Frobenius traces]{M{\"o}bius randomness law for Frobenius traces of ordinary curves}

\author{Min Sha} 
\address{School of Mathematics and Statistics, University of New South Wales, 
Sydney, NSW 2052, Australia}
\email{shamin2010@gmail.com}

\author{Igor E. Shparlinski} 
\address{School of Mathematics and Statistics, University of New South Wales, 
Sydney, NSW 2052, Australia}
\email{igor.shparlinski@unsw.edu.au}
\urladdr{\url{http://web.maths.unsw.edu.au/~igorshparlinski/}}

\subjclass[2010]{11B37, 11G20, 11J25, 11J86, 11L07}

\keywords{M{\"o}bius randomness law, smooth projective curve, Frobenius trace, Frobenius angle}

\begin{abstract} 
Recently E.~Bombieri and N.~M.~Katz (2010) have demonstrated 
that several well-known results about the distribution of values of 
linear recurrence sequences lead to interesting statements for 
Frobenius traces of algebraic curves. Here we continue this line 
of study and establish the M{\"o}bius randomness law quantitatively for 
the normalised form of Frobenius traces. 
\end{abstract}

\maketitle

\section{Introduction}

\subsection{Background on Frobenius traces}
Throughout the paper, $\cC$ denotes a smooth projective curve over a finite field $\F_q$ of 
$q$ elements. Following Bombieri and Katz~\cite{BoKa},
we consider the sequence $A_\cC (n)$ of \textit{Frobenius traces} defined by 
$$
\# \cC(\F_{q^n}) = q^n + 1 - A_\cC (n)
$$
where $\# \cC(\F_{q^n})$ is the cardinality of the set   $ \cC(\F_{q^n})$ of $\F_{q^n}$-rational points on $\cC$.

Let $g$ be the genus of $\cC$, and assume $g \ge 1$.
Since by the Weil bound (see~\cite[Section~VIII.5.9]{Lor}), we have 
$$
|A_\cC (n)| \le 2g q^{n/2}, 
$$
it is convenient to normalise the sequence $A_\cC (n)$ as 
\begin{equation}
\label{eq:an}
a_\cC (n) = \frac{A_\cC (n)}{2g q^{n/2}} \in [-1,1], 
\end{equation}
which is called the \textit{normalised Frobenius trace}.

\subsection{Some previous results} 
Here we recall some previous results on the distribution 
of the sequence $a_\cC (n)$ given by~\eqref{eq:an}. First we recall that  Bombieri and Katz~\cite{BoKa}, 
using an interpretation of $A_\cC (n)$ as a linear recurrence sequence of 
order $2g$, have showed that $|a_\cC (n)|$ is not too small.  
More precisely,  by~\cite[Theorem~3.1]{BoKa}, for any $\varepsilon> 0$ there is a constant $c(\varepsilon) > 0$ 
depending only on $\varepsilon$ such that  for every $n$ either $a_\cC (n) = 0$ or $|a_\cC (n)| \ge c(\varepsilon) g^{-1} q^{-n\varepsilon}$.  
However, unless $g=1$ (that is, $\cC$ is an elliptic curve),  the constant $c(\varepsilon)$
is not effectively computable. Using the argument of~\cite[Theorem~2.6]{EvdPSW}, one can 
give a stronger and fully effective bound of the form  
$$
|a_\cC (n)| \ge  n^{- \psi(n)}, 
$$
which for any function $\psi(z) \to \infty$ as $z \to \infty$, holds for almost all $n$ in the sense 
of asymptotic density. 

The asymptotic  distribution of the values $a_\cC (n)$ in the interval $[-1,1]$ has been studied 
in~\cite{AhShp}, where it is shown that this distribution differs from the usually occurring 
Sato--Tate law~\cite{Katz,KaSa}.  On the other hand, an asymptotic formula for the average $\ell$-adic order of 
$A_\cC (n)$ for a prime $\ell \nmid q$ has  been given in~\cite[Theorem~4]{vdPShp}. 

Several related results about the distribution of Kloosterman and Birch sums have recently been given by Perret-Gentil~\cite{P-G}.

We also note that using the upper bound of van der Poorten and 
Schlickewei~\cite[Theorem~1]{vdPSchl},
on the number of zeros of linear recurrence sequences, one can estimate the number of zero
values $a_\cC (n)=0$ in a better way which is outlined in~\cite[Section~5]{BoKa} (via uniform 
bounds on the number of zeros of linear recurrence sequences such as in~\cite{AmVia,ESS}). 

The second part of our motivations comes from the so-called \textit{M{\"o}bius randomness law}  
(see, for example,~\cite[Section~13.1]{IwKow}, and  also Sarnak's conjecture~\cite{Sarnak}) which 
 roughly asserts that for any bounded sequence $s(n)$ of complex numbers, 
 defined in terms which are not directly related to $\mu(n)$, 
we have 
$$
\sum_{n=1}^N\mu(n)  s(n) = o(N), \quad as \,\, N \to \infty.
$$ 
Here, we establish quantitatively the M{\"o}bius randomness law for 
the sequence $a_\cC (n)$ defined in~\eqref{eq:an}.

\subsection{Our results}

We recall that the M{\"o}bius function
 is  defined as $\mu(n)=0$ if an integer $n$ is divisible by a
prime squared and $\mu(n)=(-1)^r$ if $n$ is a product
of $r$ distinct primes.

Our main result is  Theorem~\ref{thm:MobExp1}. 
For the completeness, we also record Theorem~\ref{thm:MobExp0}. 

\begin{theorem}
\label{thm:MobExp0} 
For any  $B > 0$ and for  any integer $N \ge 2$, we have
$$
\left|\sum_{n=1}^N\mu(n)  a_\cC (n)\right| 
\le c(B) N (\log N)^{-B},  
$$
where $c(B) > 0$ is a constant depending only on $B$.   
\end{theorem}

When $\cC$ is an \textit{ordinary curve}, we can get a better result. 
Recall that $\cC$ is called ordinary if and only if the number of $p$-torsion points on  the Jacobian of $\cC$ is exactly $p^g$ , 
where $p$ is the characteristic of $\F_q$ (see~\cite[Definition~3.1]{Howe} for some equivalent definitions). 

Recall that the
assertion $U \ll V$  is equivalent to the inequality $|U| \le cV$ with some {\it absolute\/} constant $c>0$.

\begin{theorem}
\label{thm:MobExp1} 
If $\cC$ is an ordinary curve of genus $g \ge 1$, for any integer $N \ge 2$ we have
$$
\sum_{n=1}^N\mu(n)  a_\cC (n) \ll N^{1-1/\gamma(q,g)} (\log N)^4, 
$$
where 
$$
 \gamma(q,g) = 2^{33}3^3 \pi g^3 (\pi + \log q) \log(16g)  + 4. 
$$ 
\end{theorem}

We remark that the constant $c(B)$ in  Theorem~\ref{thm:MobExp0} is currently not effectively computable.  
For the implied constant in Theorem~\ref{thm:MobExp1}, it is effectively computable 
(due to the effectiveness of Lemma~\ref{lem:MobExp2}), but getting an explicit value of this constant is beyond the scope of this paper. 

To prove the above results, we first interpret the normalised Frobenius traces
as linear recurrence sequences via the zeta function (see the equation~\eqref{eq:an-alpha}). 
Then, Theorem~\ref{thm:MobExp0} is an immediate application of 
a result of Davenport~\cite{Davenport}, given by~\eqref{eq:Davenport} below,
on bounds on exponential sums with $\mu(n)\exp(2\pi i \alpha)$ for $\alpha \in  \R$, 
and Theorem~\ref{thm:MobExp1} is a consequence of an improvement of~\eqref{eq:Davenport}   
when some Diophantine properties of $\alpha$ are known 
(see Lemma~\ref{lem:mu-alpha}). 
More precisely,  this improvement relies on a lower bound of the denominator in 
Dirichlet's approximation of an irrational number $\alpha$ 
when $\exp(2\pi i \alpha)$ is an algebraic number (see Lemma~\ref{lem:appro-s}).
Here, we are able to take advantage of this stronger bound  in the case of Theorem~\ref{thm:MobExp1} (see Lemma~\ref{lem:Frob-ang}).

\section{Preliminaries}

\subsection{Linear form in the logarithms of algebraic numbers}

The main tool in this paper is Baker's theory of linear forms in the logarithms of algebraic 
numbers, see~\cite{Bug}. 
Here we restate one of its explicit forms due to Baker and W{\"u}stholz~\cite{BW}. 

First, recall that for a non-zero complex number $z$, the principal value of the natural logarithm of $z\in \C$ is 
$$
\log z = \log |z| + i \cdot \Arg(z),
$$
where as usual, $i$ is the imaginary unit, and
 $\Arg(z)$ is the principal value of the arguments of $z$ ($ 0 \le \Arg(z) < 2\pi$). 

Let
$$
\Lambda = b_1 \log \alpha_1 + b_2 \log \alpha_2 + \cdots + b_n \log \alpha_n,
$$
where $n \ge 2$, $b_1,\ldots,b_n \in \Z$, and $\alpha_1,\ldots, \alpha_n$ are non-zero elements of a number field $K$. 
Let $d =[K:\Q]$ and $B=\max \{ |b_1|,\ldots, |b_n| \}$. For all $1 \le j \le n$,  choose a real number $A_j$ such that
$$
A_j \ge \max \{ \h(\alpha_j), |\log \alpha_j |/d, 1/d \}, 
$$
where $\h$ stands for the  logarithmic absolute  Weil height 
(note that the height function used in~\cite{BW} is different from ours). 

Suppose that $\Lambda \ne 0$. Then, we have
\begin{equation} \label{eq:BW}
\log |\Lambda| > -C(n,d)A_1\cdots A_n \cdot \max\{\log B, 1/d\}, 
\end{equation}
where 
$$
C(n,d) = 18(n+1)!n^{n+1}(32d)^{n+2}\log(2nd). 
$$

We remark that we in fact only need a lower bound on linear forms in two logarithms of the form in~\cite[Th{\'e}or{\`e}me~3]{LMN}. 
However, the lower bound in~\cite[Th{\'e}or{\`e}me~3]{LMN} essentially has the term $(\log B)^2$. 
This is not sufficient for our purpose. 
Alternatively, using \cite[Theorem 2.1]{Gou}  one may obtain such a lower bound having the term $\log B$. 

In addition, there is another explicit form of Baker's theory due to Matveev~\cite[Corollary~2.3]{Matveev}. 
However, for our purpose the one of Baker and W{\"u}stholz~\cite{BW} gives a slightly better result.

\subsection{Lower bounds for Diophantine approximations}

The famous theorem of Roth~\cite{Roth} states that 
given an irrational algebraic number $\alpha$ and $\varepsilon>0$, 
there exists a constant $c(\alpha, \varepsilon) > 0$ such that 
for any integers $r,s$ with $s>0$ we have 
\begin{equation}  \label{eq:Roth}
\left|  \alpha - \frac{r}{s} \right| > \frac{c(\alpha, \varepsilon)}{s^{2 + \varepsilon}}. 
\end{equation}
Certainly, in general for real transcendental numbers no such bound is possible. 
Here,   using Baker's theory of  linear forms in logarithms (see~\cite{Bug}), 
we obtain such a lower bound for a special kind of real transcendental numbers, namely, 
for irrational arguments of algebraic numbers.  

Now, let $\alpha$ be an irrational number. Define 
$$
\e(\alpha) = e^{2\pi i \alpha}, 
$$
where as usual $e$ is the base of the natural logarithm. 
If $\e(\alpha)$ is an algebraic number, 
then by the Gelfond--Schneider theorem we know that $\alpha$ must be a transcendental number. 
Indeed, assume that $\alpha$ is an algebraic number, then by the Gelfond--Schneider theorem 
$1 = \e(\alpha)^{1/\alpha}$ is a transcendental number, which is impossible. 

The following result gives such a lower bound for any irrational number $\alpha$ when $\e(\alpha)$ is an algebraic number. 
This can be viewed as a Diophantine property of the arguments of algebraic numbers. 
In fact, $1+ \kappa(\alpha)$ is an upper bound of the {\it irrationality exponent\/} of $\alpha$. 

We also remark that the following result is essentially a variant of~\cite[Theorem~4.1]{BoKa}. 
Following the same strategy, we provide a proof for the convenience of the reader and also put it 
in a form suitable for our applications. 

\begin{lemma}  \label{lem:arg-appro}
Let $\alpha$ be an irrational number. 
Assume that $\e(\alpha)$ is an algebraic number. 
Then,  for any integers $r,s$ with $s \ge 1$, we have 
$$  
\left|  \alpha - \frac{r}{s} \right| > \frac{1}{\pi(2s)^{1+\kappa(\alpha)}}, 
$$
where 
\begin{align*}
& \kappa(\alpha) = 2^{25} 3^3 \pi d^3 A_1 \log(4d)  , \\
& d = [ \Q(\e(\alpha)) : \Q], \\
& A_1 = \max\{\h(\e(\alpha)), 2\pi \alpha/d, 1/d\}.\\
\end{align*}
\end{lemma}

\begin{proof} 
We can always replace $\alpha$ with its fractional part $\{\alpha\}$. Hence, 
without loss of generality, we assume that $0 < \alpha < 1$ and $\gcd(r,s)=1$.  
Then, if $|r| > s$, we have 
$$
\left|  \alpha - \frac{r}{s} \right| > \frac{1}{s},
$$
which is better than the desired result. 
In the sequel, we assume $|r| \le s$. 

Denote 
$$
 \Delta = \alpha - \frac{r}{s}. 
$$
Then, since $0 < \alpha < 1$, we have 
$$
\log \e(\alpha) = 2\pi i \alpha = 2\pi i (\Delta + \frac{r}{s}), 
$$
and so 
$$
2s\pi i \Delta = s\log \e(\alpha) - 2r\pi i = s\log \e(\alpha) - 2r\log (-1). 
$$
Denote 
$$
\Lambda = s\log \e(\alpha) - 2r\log (-1). 
$$
So, 
\begin{equation} \label{eq:Del}
|\Delta| = \frac{|\Lambda|}{2\pi s}. 
\end{equation}

Note that $\alpha$ is an irrational number. 
So, $\e(\alpha)$ is not a root of unity, and thus $\Lambda \ne 0$.   
Using~\eqref{eq:BW} with $n=2$, we obtain 
\begin{equation}   \label{eq:Lam}
\log |\Lambda| > -C(d)A_1A_2 \cdot \max\{\log B, 1/d\}, 
\end{equation}
where 
\begin{align*}
& C(d) = C(2,d) = 2^{25} 3^3d^4\log(4d), \\
& d = [ \Q(\e(\alpha)) : \Q], \\
& A_1 = \max\{\h(\e(\alpha)), 2\pi\alpha/d, 1/d\}, \quad A_2 = \pi/d, \\
& B = \max\{s, 2|r| \}. 
\end{align*}
Since $|r| \le s$, we have $B \le 2s$. 
In view of $s \ge 1$ and $d \ge 2$, we get 
$$
\max\{\log B, 1/d\} \le \log(2s).
$$
Hence, the inequality~\eqref{eq:Lam} becomes 
$$
\log |\Lambda| > - 2^{25} 3^3 \pi d^3\log(4d)  A_1 \log (2s),
$$
which, together with~\eqref{eq:Del}, implies the desired result. 
\end{proof}

\subsection{Dirichlet's theorem} 

We first recall {\it Dirichlet's theorem\/} in Diophantine approximation; see, for example~\cite[Equation~(20.29)]{IwKow}. 

\begin{lemma} \label{lem:Dirichlet}
Let $\alpha$ be an irrational number. 
Then, for any integer $N \ge 2$ there are two integers $r,s$ such that 
$$
0 < \left| \alpha - \frac{r}{s} \right| \le \frac{1}{sN}, \quad 1 \le s \le N, \ \gcd(r,s)=1. 
$$
\end{lemma}

In Lemma~\ref{lem:Dirichlet}, if $N$ tends to infinity, then $s$ also goes to infinity. 
It is natural to ask how large $s$ can be. 
If $\alpha$ is an irrational algebraic number, combining Lemma~\ref{lem:Dirichlet}  with the bound~\eqref{eq:Roth} we have 
\begin{equation}  \label{eq:scN}
s > \big( c(\alpha, \varepsilon)N \big)^{1/(1+\varepsilon)}. 
\end{equation}

The next result follows directly from Lemmas~\ref{lem:arg-appro} and~\ref{lem:Dirichlet}. 

\begin{lemma}  \label{lem:appro-s}
Let $\alpha$ be an irrational number. 
Assume that $\e(\alpha)$ is an algebraic number. 
Then,  for any  integer $N \ge 2$, there are integers $r,s$ such that 
$$
0 < \left| \alpha - \frac{r}{s} \right| \le \frac{1}{sN}, \quad 1 \le s \le N, \ \gcd(r,s)=1, 
$$
and 
$$  
s > \frac{1}{2}\big( N/(2\pi) \big)^{1/\kappa(\alpha)},
$$
where $\kappa(\alpha)$ has been defined in Lemma~\ref{lem:arg-appro}.
\end{lemma}

\subsection{Exponential sums with  M{\"o}bius function}

Recall the following bound of exponential sums with M{\"o}bius function, 
which depends on the Diophantine 
properties of the exponent $\alpha$; see~\cite[Theorem~13.9]{IwKow}. 
It can be viewed as a variant of the modern form of the Vinogradov bound
for exponential sums over primes;  see~\cite[Section~13.5]{IwKow} for more details. 

\begin{lemma}
\label{lem:MobExp2} 
Suppose that the real $\alpha$ satisfies
$$
\left|\alpha - \frac{r}{s}\right| \le \frac{1}{s^2} 
$$
for some integers $r, s$ with $s>0$ and $\gcd(r,s) = 1$.
Then, for any integer $N \ge 2$, we have 
$$
\sum_{n=1}^N\mu(n)  \e(n\alpha)
\ll \left(s^{1/4} N^{1/4} + s^{-1/4} N^{1/2} + N^{2/5}\right) N^{1/2}(\log N)^4.
$$
\end{lemma}

We remark that the implied constant in Lemma~\ref{lem:MobExp2} is effectively computable, 
but getting an explicit value of this constant is beyond the scope of this paper. 

In addition, Davenport~\cite{Davenport} (see also~~\cite[Theorem~13.10]{IwKow}) has established the following general result: 
for any real number $\alpha$ and $N \ge 2$, we have 
\begin{equation}  \label{eq:Davenport}
\left|\sum_{n=1}^N\mu(n)  \e(n\alpha)\right| \le c(B) N(\log N)^{-B}
\end{equation}
for any $B > 0$, where $c(B) > 0$ is a constant depending only on $B$. 
We remark that the constant $c(B)$ is currently not effectively computable. 
The upper bound in~\eqref{eq:Davenport} has a very attractive feature that it is independent of $\alpha$. 

Involving the dependence on $\alpha$, the bound~\eqref{eq:Davenport} can be improved for some special cases. 

\begin{lemma}  \label{lem:mu-alpha}
Let $\alpha$ be an irrational number. 
Assume that $\e(\alpha)$ is an algebraic number. 
Then,  for any  integer $N \ge 2$, we have 
$$  
\sum_{n=1}^N\mu(n)  \e(n\alpha)
\ll N^{1-1/(4\kappa(\alpha) +4)} (\log N)^4, 
$$
where $\kappa(\alpha)$ has been defined in Lemma~\ref{lem:arg-appro}.
\end{lemma}

\begin{proof}
By Lemma~\ref{lem:appro-s}, for any integer $M \ge 2$, there are integers $r,s$ such that 
$$
0 < \left| \alpha - \frac{r}{s} \right| \le \frac{1}{s M}, \quad 1 \le s \le M , \  \gcd(r,s)=1, 
$$
and 
$$
s > \frac{1}{2} \big( M /(2\pi) \big)^{1/\kappa(\alpha)},
$$ 
where $\kappa(\alpha)$ has been defined in Lemma~\ref{lem:arg-appro}. 
Then, by Lemma~\ref{lem:MobExp2}, we have 
\begin{equation} \label{eq:muN}
\sum_{n=1}^N\mu(n)  \e(n\alpha)
\ll \left(s^{1/4} N^{1/4} + s^{-1/4} N^{1/2} + N^{2/5}\right) N^{1/2}(\log N)^4.
\end{equation} 
Note that 
$$
\frac{1}{2}  \big( M/(2\pi) \big)^{1/\kappa(\alpha)} < s \le   M, 
$$
hence 
$$
s^{1/4} N^{1/4} + s^{-1/4} N^{1/2} \ll M^{1/4} N^{1/4} +  M^{- 1/(4\kappa(\alpha))} N^{1/2}, 
$$
which with 
$$
M = \rf{N^{\kappa(\alpha) / (\kappa(\alpha) + 1)}} \ge 2
$$
becomes 
$$
s^{1/4} N^{1/4} + s^{-1/4} N^{1/2} \ll  N^{1/2 - 1/(4\kappa(\alpha) +4)}.
$$  
Substituting this into~\eqref{eq:muN},  we see that the term  $N^{9/10}$ never dominates, and 
we obtain  the desired result.  
\end{proof}

From the above proof, one can see that if an irrational number $\alpha$ has a Diophantine property 
as in Lemma~\ref{lem:appro-s}, then the upper bound~\eqref{eq:Davenport} can be improved similarly. 
For example, this can be done for irrational algebraic numbers by~\eqref{eq:scN}.

\subsection{Frobenius eigenvalues and angles}

We refer to~\cite{Lor} for a background on curves and their zeta-functions.

For  a smooth projective curve $\cC$  over the finite field $\F_q$, we define the 
zeta-function of $\cC$ as 
$$
Z(T)=\exp\left(\sum_{n=1}^{\infty}\#\cC(\F_{q^n})\frac{T^n}{n}\right). 
$$
It is well-known that if $\cC$ is of genus $g \ge 1$ then 
$$
Z(T)=\frac{P(T)}{(1-T)(1-qT)},
$$
where 
$$
P(T)=\prod_{j=1}^{2g}(1-\beta_j T) 
$$
 is a polynomial of degree $2g$ with integer coefficients, and 
  $\beta_1, \beta_2, \ldots, \beta_{2g}$ are algebraic integers, 
called the {\it Frobenius  eigenvalues\/}, which  satisfy  
\begin{equation}
\label{eq:Frob}
|\beta_j|  = q^{1/2},  \qquad j =1, 2, \ldots, 2g;
\end{equation}
see~\cite[Section~VIII.5]{Lor}. 
Then,  for each $\beta_j$, since all its conjugates have absolute value $q^{1/2}$, 
we have (via the Mahler measure, see \cite[Lemma 3.10]{Waldschmidt}) 
\begin{equation}  \label{eq:height}
\h(\beta_j) = \frac{1}{2} \log q,  \qquad j =1, 2, \ldots, 2g.
\end{equation}

Furthermore, in view of~\eqref{eq:Frob} we write 
\begin{equation}
\label{eq:beta} 
\beta_j =  q^{1/2} \e(\alpha_j), 
\end{equation} 
with some $\alpha_j \in [0,1)$, $j =1, 2,\ldots, 2g$. 
Usually, these $2\pi \alpha_j$ are called \textit{Frobenius angles}. 
We then call $\alpha_j$ \textit{normalised Frobenius angles}. 
Now simple combinatorial arguments lead to the  well-known identity
$$
  \# \cC(\F_{q^n}) = q^n + 1 - \sum _{j=1}^{2g}\beta_j^n, 
$$
which implies
\begin{equation}
\label{eq:an-alpha}
a_\cC (n) = \frac{1}{2g}  \sum _{j=1}^{2g}\e(n\alpha_j).
\end{equation}
This is crucial for our approach.

\subsection{Diophantine properties of normalised Frobenius angles}

We recall the following irrationality property of normalised Frobenius angles 
given by~\cite[Lemma~8]{AhShp}

 \begin{lemma}
\label{lem:Irrat}
Suppose that $\cC$ is an ordinary smooth projective curve  of genus $g \ge 1$ 
over $\F_q$. Then all its normalised Frobenius angles  $\alpha_j$, $ j =1, 2, \ldots, 2g$, are irrational. 
\end{lemma}

Now, we can use Lemma~\ref{lem:arg-appro} to obtain a Diophantine property for the normalised Frobenius angles. 

\begin{lemma}
\label{lem:Frob-ang}
Suppose that $\cC$ is an ordinary smooth projective curve  of genus $g \ge 1$ 
over $\F_q$. Let $\alpha$ be an arbitrary normalised Frobenius angle of $\cC$. 
Then,  for any integers $r,s$ with $s \ge 1$, we have 
$$  
\left|  \alpha - \frac{r}{s} \right| > \frac{1}{\pi (2s)^{1+\kappa(q,g)}}, 
$$
where 
$$
 \kappa(q,g) = 2^{31}3^3 \pi g^3 (\pi + \log q) \log(16g) . 
$$
\end{lemma}

\begin{proof}
Let $\beta$ be the Frobenius eigenvalue corresponding to $\alpha$ as defined in~\eqref{eq:beta}. 
That is, $\beta = q^{1/2}\e(\alpha)$. 
Hence, $\e(\alpha)$ is an algebraic number.  
Besides, by Lemma~\ref{lem:Irrat}, $\alpha$ is an irrational number. 
Then, appying Lemma~\ref{lem:arg-appro}, we have that 
for any integers $r,s$ with $s \ge 1$
$$  
\left|  \alpha - \frac{r}{s} \right| > \frac{1}{\pi (2s)^{1+\kappa(\alpha)}}, 
$$
where 
\begin{align*}
& \kappa(\alpha) = 2^{25} 3^3 \pi d^3\log(4d)  A_1, \\
& d = [\Q(\e(\alpha)) : \Q], \\
& A_1 = \max\{\h(\e(\alpha)), 2\pi \alpha/d, 1/d\}.\\
\end{align*}
Since $\deg \beta \le 2g$, we have 
$$
d = [\Q(\e(\alpha)) : \Q] = [\Q(\beta q^{-1/2}) : \Q] \le 2\deg \beta \le 4g. 
$$
Using~\eqref{eq:height}, we obtain 
$$
\h(\e(\alpha)) = \h(\beta q^{-1/2}) \le \h(\beta) + \h(q^{1/2}) = \log q. 
$$
Note that we must have $d \ge 2$. 
So, we have 
$$
A_1  \le \pi + \log q. 
$$
Hence, we get 
$$
\kappa(\alpha) \le 2^{31}3^3 \pi g^3 (\pi + \log q) \log(16g). 
$$
This completes the proof. 
\end{proof}

\section{Proofs of the main results}

\subsection{Proof of Theorem~\ref{thm:MobExp0}}

By~\eqref{eq:an-alpha}, we have 
$$
\left|\sum_{n=1}^N\mu(n)  a_\cC (n)\right| 
 \le  \frac{1}{2g} \sum _{j=1}^{2g} \left|\sum_{n=1}^N\mu(n)  \e(n\alpha_j)\right|. 
$$
Then, the desired result follows directly from the bound~\eqref{eq:Davenport}.

\subsection{Proof of Theorem~\ref{thm:MobExp1}}

Let $\alpha_1, \alpha_2, \ldots, \alpha_{2g}$ be the normalised Frobenius angles of $\cC$. 
For each $\alpha_j$, using Lemmas~\ref{lem:mu-alpha} and~\ref{lem:Frob-ang}, we have that 
for any  integer $N \ge 2$, 
\begin{equation}  \label{eq:mu-alpha}
\sum_{n=1}^N\mu(n)  \e(n\alpha_j)
\ll N^{1-1/(4\kappa(q,g) +4)} (\log N)^4, 
\end{equation}
where $\kappa(q,g)$ has been defined in Lemma~\ref{lem:Frob-ang}.

So, combining~\eqref{eq:an-alpha} with~\eqref{eq:mu-alpha}, 
for any integer $N \ge 2$, we have 
\begin{align*}
\left|\sum_{n=1}^N\mu(n)  a_\cC (n)\right| 
& = \frac{1}{2g}\left|\sum_{n=1}^N\mu(n)  \sum _{j=1}^{2g}\e(n\alpha_j)\right| \\
& \le \frac{1}{2g}\sum _{j=1}^{2g} \left|\sum_{n=1}^N\mu(n)  \e(n\alpha_j)\right| \\
& \ll N^{1-1/(4\kappa(q,g)+4)} (\log N)^4.  
\end{align*}
This completes the proof.

\section{Comments}
We remark that our methods also apply to  the sequences  of  character sums 
$$
{\mathcal S}_R(n)  =  \frac{1}{q^{n/2}} \sum_{x \in \F_{q^n}^*} \psi\left( \Tr_{\F_{q^n}/\F_q}(R(x)) \right),
$$ 
where  $\Tr_{\F_{q^n}/\F_q}$ is the trace function from $\F_{q^n}$ to $\F_q$, $\psi$ is a fixed additive character of $\F_q$ 
and $R(X) \in \F_q(X)$. 
We recall that if  $R(X)$ is not of the form $R(X) = Q(X)^p - Q(X)$ for any
other function $Q(X) \in \F_q(X)$, where $p$ is the characteristic of $\F_q$,
 then  ${\mathcal S}_R(n)$ can be expressed as a  power sum of the same type as~\eqref{eq:an-alpha};  
see~\cite[Section~11.11]{IwKow}. Hence a full analogue of   Theorem~\ref{thm:MobExp0}  holds for these sums for any $R(X) \in \F_q(X)$, such that  
$R(X)  \ne Q(X)^p - Q(X)$ for any $Q(X) \in \F_q(X)$.

The situation with an   analogue of Theorem~\ref{thm:MobExp1} is more complicated as one 
need an irrationality of angles statement 
similar to  Lemma~\ref{lem:Irrat}, which does not seem to be readily available.
There are some results in this direction in~\cite{BoKa,P-G}, in particular showing 
that the required irrationality holds ``generically'' in some special cases , but they  are not enough
to make a conclusive statement about  the sums  ${\mathcal S}_R(n)$. Hence, we pose this as an question to find appropriate 
sufficient conditions on $R(X)$ for which an analogues of  Theorem~\ref{thm:MobExp1} 
holds for   ${\mathcal S}_R(n)$. 

 Perhaps, it is natural to start with  Kloosterman sums 
$$
\cK(n) = \frac{1}{q^{n/2}} \sum_{x \in \F_{q^n}^*} \psi\(\Tr_{\F_{q^n}/\F_q}\(ax + x^{-1}\)\), \qquad n =1,2, \ldots, 
$$ 
where   $a\in \F_q^*$ is a fixed element. 
Indeed, it is well known that for some $\vartheta \in \C$ with $|\vartheta| =q^{1/2}$  we have 
$$
\cK(n) =  \vartheta^n + \overline \vartheta^n, \qquad n =1,2, \ldots; 
$$ 
see~\cite{BoKa,P-G}, where this representation is also exploited. 
We note that this setting is dual to that of~\cite{FKM,KorShp} where Kloosterman  
sums modulo a large fixed prime $p$ (and more general trace functions)  are ordered by the coefficient $n$ in the  exponent $nx + x^{-1}$, rather than by the degree of the field extension as in this setting.

\section*{Acknowledgement}

The authors are very grateful to the referee for valuable comments.   

During the preparation of this paper, the first author was supported by 
the Australian Research Council Grant~DE190100888, and  the second author was partially supported
by the Australian Research Council Grant~DP180100201.


\begin{thebibliography}{99}   

\bibitem{AhShp}
O. Ahmadi and I. E. Shparlinski, 
\textit{On the distribution of the number of points on algebraic curves in extensions of finite fields},
Math. Res. Lett., 17 (2010), 689--699.

\bibitem{AmVia}
F. Amoroso and E. Viada, \textit{On the zeros of linear recurrence sequences}, 
Acta Arith., 147 (2011), 387--396.

\bibitem{BW}
A. Baker and G. W{\"u}stholz,  \textit{Logarithmic forms and group varieties}, J. Reine Angew. Math., 442 (1993), 19--62.


\bibitem{BoKa} 
E. Bombieri and N. M. Katz, \textit{A note on  lower bounds for Frobenius traces}, L'Enseignement Math{\'e}matique, 56 (2010), 203--227.

\bibitem{Bug} Y. Bugeaud,  \textit{Linear forms in logarithms and applications},  IRMA Lectures in Math. and Theor.  Physics, 
vol.~28, European Math.  Soc., Z{\"u}rich, 2018. 


\bibitem{Davenport}
H. Davenport, \textit{On some infinite series involving arithmetical functions II}, Quart. J. Math., 8 (1937), 313--320.


\bibitem{EvdPSW} 
G. Everest, A. J. van der Poorten, I. E. Shparlinski 
and T. B. Ward, \textit{Recurrence sequences}, Amer. Math. Soc., 2003.

\bibitem{ESS} 
J.-H. Evertse,  H. P.  Schlickewei and W. M. Schmidt, 
\textit{Linear equations in variables which lie in a multiplicative group},
Ann. of Math., 155 (2002), 807--836. 

\bibitem{FKM} 
{\'E}. Fouvry,  E. Kowalski and  P.  Michel, \textit{Algebraic trace functions over the primes},
Duke Math. J.,  163 (2014),  1683--1736.

\bibitem{Gou}
N. Gouillon, \textit{Explicit lower bounds for linear forms in two logarithms}, 
J. Th{\'e}or. Nombres Bordeaux, 18 (2006), 125--146.

\bibitem{Howe}
E. Howe, \textit{Principally polarized ordinary abelian varieties over finite fields}, 
Trans. Amer. Math. Soc., 347 (1995), 2361--2401.

\bibitem{IwKow} 
H. Iwaniec and E. Kowalski,
\textit{Analytic number theory}, Amer.  Math.  Soc., Providence, RI, 2004.

\bibitem{Katz}
N. M. Katz, \textit{Convolution and equidistribution:  Sato--Tate theorems for finite-field Mellin transforms}, 
Annals of Math. Studies,  vol.~80, Princeton Univ. Press, Princeton, NJ, 2012.

\bibitem{KaSa}
N. M. Katz and P. Sarnak, \textit{Random matrices, Frobenius eigenvalues, and monodromy}, 
Amer. Math. Soc. Colloquium Publications, vol.~45, Amer. Math. Soc., Providence, RI, 1999.

\bibitem{KorShp} 
M. A. Korolev and I. E. Shparlinski, 
\textit{Sums of  algebraic trace functions  twisted by arithmetic   functions}, 
Pacific J. Math., 304 (2020), 505--522. 


\bibitem{LMN}
M. Laurent, M. Mignotte and Y. Nesterenko, 
\textit{Formes lin{\'e}aires en deux logarithmes et d{\'e}terminants d'interpolation}, J. Number Theory, 55 (1995), 285--321.

\bibitem{Lor}
D. Lorenzini, \textit{An invitation to arithmetic geometry},
Amer. Math. Soc., 1996.

\bibitem{Matveev}
E.M. Matveev, \textit{An explicit lower bound for a homogeneous rational linear form in the logarithms of algebraic numbers II}, Izv. Math., 64 (2000), 1217--1269.


\bibitem{P-G} 
C.  Perret-Gentil,
\textit{Roots of $L$-functions of characters over function fields, generic linear independence and biases},  
Algebra \& Number Theory, to appear, available at \url{https://arxiv.org/abs/1903.05491}.  

\bibitem{vdPSchl}  
A. J. van der Poorten and H. P. Schlickewei, 
\textit{Zeros of recurrence sequences}, Bull. Austral Math. Soc., 
44 (1991), 215--223.

\bibitem{vdPShp}  
A. J. van der Poorten and I. E. Shparlinski, \textit{On the number of zeros of exponential polynomials
and related questions},  Bull. Austral Math. Soc., 46 (1992), 401--412.

\bibitem{Roth}
K.F. Roth, \textit{Rational approximations to algebraic numbers}, Mathematika, 2 (1955), 1--20.

\bibitem{Sarnak}
P. Sarnak, \textit{Three lectures on the M{\"o}bius function, randomness and dynamics}, 2011, 
available at \url{https://publications.ias.edu/node/512}.  

\bibitem{Waldschmidt}
M. Waldschmidt, \textit{Diophantine approximation on linear algebraic groups}, 
Grundlehren Math. Wiss., vol. 326, Springer, Berlin, 2000.

\end{thebibliography}
\end{document}